\documentclass{conm-p-l}

\usepackage{amssymb}

\usepackage[cmtip,all]{xy}

\usepackage{tikz-cd}

\newtheorem{theorem}{Theorem}[section]
\newtheorem{lemma}[theorem]{Lemma}

\theoremstyle{definition}

\theoremstyle{remark}
\newtheorem{remark}[theorem]{Remark}

\numberwithin{equation}{section}

\begin{document}

\title{Cobordism of $G$-manifolds}


\author{Jack Carlisle}
\address{Department of Mathematics, the University of Notre Dame, South Bend, Indiana 46616}
\curraddr{}
\email{jcarlisl@nd.edu}
\thanks{}


\subjclass[2020]{Primary 55-02, 55N22, 55P91}

\keywords{Equivariant cobordism}

\date{June 30, 2023}

\begin{abstract}
We survey some results in the field of equivariant cobordism. In particular, we use methods from equivariant stable homotopy theory to calculate the unoriented $C_2$-equivariant bordism ring $\Omega^{C_2}_*$, which was originally calculated by Alexander using other methods. Our proof method generalizes well to other settings, such as equivariant complex cobordism, and affords a formal group theoretic interpretation of Alexander's calculation.
\end{abstract}

\maketitle

\tableofcontents

\section*{Introduction}

Suppose $M_1$ and $M_2$ are smooth, closed manifolds. We say $M_1$ is {\it cobordant} to $M_2$ if there exists a smooth manifold $W$ whose boundary is 
\[\partial W \cong M_1 \sqcup M_2.\]
Equivalence classes of manifolds under this equivalence relation form the unoriented cobordism ring
\[
\Omega_* = \dfrac{\{\text{Smooth closed manifolds}\}}{\text{cobordism}}.
\]
While there are important variants of this construction such as oriented cobordism and complex cobordism, we restrict our attention to the unoriented case. The ring $\Omega_*$ is calculable because it arises as the coefficient ring of a spectrum, namely the Thom spectrum $MO$. As an orthogonal spectrum, the component spaces
\[
(MO)_n = \text{Thom}\left(\gamma \to BO(n) \right)
\] 
of $MO$ are the Thom complexes of universal vector bundles. The Pontrjagin-Thom construction determines a ring map
\[ \begin{tikzcd} 
\Omega_* \ar[r] & MO_*,
\end{tikzcd} \]
which turns out to be an isomorphism. This identification allows one to calculate $\Omega_*$ using techniques from stable homotopy theory. More precisely, Thom \cite{Thom} proved that the unoriented cobordism ring $\Omega_*$ is given by 
\[
\Omega_* = \mathbb{F}_2[x_i : i \neq 2^n -1],
\]
for certain classes $x_i$ of degree $i$.

Suppose now that $G$ is a finite group. We wish to understand cobordism in the $G$-equivariant setting. The appropriate analogue of $\Omega_*$ is the unoriented $G$-cobordism ring 
\[
\Omega^G_* = \dfrac{\{\text{Smooth closed } G\text{ - manifolds}\}}{G\text{ - cobordism}}.
\]
In order to calculate $\Omega^G_*$, we hope to represent this ring by a $G$-spectrum. There is a natural $G$-equivariant analogue of $MO$, which denote $MO_G$. As an orthogonal $G$-spectrum, the component spaces 
\[
(MO_G)_n = \text{Thom}\left( \gamma_{G,n}\to BO_G(n) \right)
\]
of $MO_G$ are the Thom complexes of universal $G$-equivariant vector bundles. The $G$-equivariant Pontrjagin-Thom construction (which we explicate in section 2) determines a ring map 
\[ \begin{tikzcd} 
\Omega^G_* \ar[r] &  MO^G_*,
\end{tikzcd} \]
but if $G$ is a non-trivial group, then this map is {\it not} an isomorphism. This is related to the fact that transversality is not a generic property in the $G$-equivariant setting \cite{Wasserman}.

While the unoriented $G$-cobordism ring is not represented by the expected $G$-spectrum $MO_G$, it {\it is} represented by another $G$-spectrum, which we suggestively denote $\Omega_G$.\footnote{Some other sources use the notation $mO_G$ for the spectrum $\Omega_G$} As an orthogonal $G$-spectrum, the component spaces 
\[
(\Omega_G)_{n} = \text{Thom} \left( \gamma_n \to BO(n) \right)
\]
of $\Omega_G$ are the Thom complexes of universal (non-equivariant) vector bundles. While $G$ acts trivially on $(\Omega_G)_n$ for each $n \geq 0$, the component spaces $(\Omega_G)_V$ for $G$ a non-trivial $G$-representation carry a non-trivial $G$-action, so $\Omega_G$ is not a trivial $G$-spectrum. The component space $(\Omega_G)_V$ is the Thom complex of a certain Tautological vector bundle $\gamma_V$ which is, in a precise sense, intermediate between the universal vector bundle over $BO(n)$, and the universal $G$-equivariant vector bundle over $BO_G(n)$. 

The identification of $\Omega^G_*$ with the coefficients of a $G$-spectrum allows one to calculate $\Omega^G_*$ using techniques from equivariant stable homotopy theory. We will use such techniques to calculate the $C_2$-cobordism ring $\Omega^{C_2}_*$, where $C_2$ is the group of order $2$ (see Theorem \ref{zgraded}). The ring $\Omega^{C_2}_*$ was originally calculated by Alexander \cite{Alexander} using other methods. Our proof method is of independent interest, and generalizes well to other settings, such as equivariant complex cobordism \cite{Jack}. In addition to our calculation of $\Omega^{C_2}_*$, we calculate a presentation of the extended cobordism ring $\Omega^{C_2}_\diamond$ (see Theorem \ref{goodrange}), of which $\Omega^{C_2}_*$ is a subring, and of which  $MO^{C_2}_*$ is a localization. This provides algebraic insight into a theorem of Brocker and Hook \cite{BrockerHook}.

\section{Representing cobordism of $G$-manifolds}

Suppose $G$ is a finite group. Let $\Omega^G_*$ denote the unoriented $G$-cobordism ring. Elements of $\Omega^G_*$ are represented by smooth compact $G$-manifolds, and two such manifolds $M_1$ and $M_2$ represent the same class in $\Omega^G_*$ if there is a $G$-manifold $W$ whose boundary is 
\[
\partial W \cong M_1 \sqcup M_2.
\]
The addition in $\Omega^G_*$ is determined by 
\[
[M_1] + [M_2] = [M_1 \sqcup M_2],
\]
and the product in $\Omega^G_*$ is determined by 
\[
[M_1][M_2] = [M_1 \times M_2],
\]
where $G$ acts diagonally on $M_1 \times M_2$. 

The bordism rings $\Omega^H_*$ for $H \leq G$ assemble to form a $G$-Mackey functor.  In particular, this means that for any subgroup inclusion $H \leq K$, we have restriction and transfer maps
\[ \begin{tikzcd} \Omega^K_* \ar[r,"\text{res}^K_H"] &  \Omega^H_* ,
\end{tikzcd} \] 
\[ \begin{tikzcd} \Omega^H_* \ar[r,"\text{tr}^K_H"] &  \Omega^K_*.
\end{tikzcd} \] 
The restriction $\text{res}^K_H$ send the class of a $K$-manifold $M$ to the class of its underlying $H$-manifold. The transfer $\text{tr}_H^K$ sends the class of a $H$-manifold $M$ to the class of the induced $K$-manifold $K \times_H M$. We write $\underline{\Omega}^G_*$ when we wish to emphasize the Mackey functor structure of geometric cobordism.

Next, we define the $G$-spectrum $\Omega_G$ which represents geometric cobordism. We use orthogonal $G$-spectra as our model of spectra, and refer the interested reader to \cite{HHR}, section 2.2.2, for background on orthogonal $G$-spectra. If $V$ is an orthogonal $G$-representation, let  $BO(V) = \text{Gr}_{|V|}(V \oplus \mathbb{R}^\infty)$ denote the Grassmannian of $|V|$-dimensional subspaces of $V \oplus \mathbb{R}^\infty$, where $C_2$ acts trivially on $\mathbb{R}^\infty$. This $G$-space carries a tautological $G$-equivariant vector bundle $\gamma_V$, and as an orthogonal $G$-spectrum, we define $\Omega_G$ by
\[
(\Omega_G)_V =  \text{Thom}\left( \gamma_V  \to BO(V) \right).
\]
In order to define our ring homomorphism $\Omega^G_* \to \pi^G_*(\Omega_G)$, we review the $G$-equivariant Pontrjagin-Thom construction. 

Suppose we are given a class $[M] \in \Omega^G_n$ which is represented by an $n$-dimensional $G$-manifold $M$. Choose a $G$-equivariant embedding of $M$ into an orthogonal $G$-representation of the form $ \mathbb{R}^n \oplus V$. Let $S^{n + V}$ denote the one-point compactification of $ \mathbb{R}^n \oplus V$. By \cite{Wasserman}, we may choose a $G$-equivariant tubular neighborhood $N$ of $M$ in $\mathbb{R}^n \oplus V$, so that the quotient 
\[S^{ n + V} / (S^{n + V} \setminus N) \cong \text{Thom}(\nu \to M)\]
is $G$-equivariantly homeomorphic to the Thom complex of the normal bundle $\nu$ of $M$ in $\mathbb{R}^n \oplus V$. The normal bundle $\nu$ is equipped with a bundle map to the tautological vector bundle over $\text{Gr}_{|V|}(V \oplus \mathbb{R}^n)$, which itself maps to the tautological vector bundle over $BO(V)$. The corresponding composite 
\[ \begin{tikzcd} 
S^{V + n} \ar[r] & \dfrac{S^{V +n}}{S^{V+n} \setminus N} \ar[r,"\cong"] &  \text{Thom}(\nu \to M) \ar[d] \\
 & &   \text{Thom}\left(\gamma_V \to BO(V) \right) \ar[d,"="]\\
 &  &(\Omega_G)_V
\end{tikzcd} \]
represents a class in $\pi_n^G(\Omega_G)$.

Having explicated the equivariant Pontryagin-Thom construction, we may now state the following theorem, which identifies the coefficient ring of the $G$-spectrum $\Omega_G$ with the (geometrically defined) $G$-cobordism ring $\Omega^G_*$.

\begin{theorem} (\cite{Schwede2}, Theorem 6.2.33)
For any finite group $G$, the assignment 
\[
[M] \mapsto [S^{n + V} \to (\Omega_G)_V],
\]
as defined above, determines a ring isomorphism 
\[\begin{tikzcd}
\Omega^G_* \ar[r,"\cong"] & \pi^G_*(\Omega_G).
\end{tikzcd} \] 
In fact, this holds true at the level of $H$-fixed points for any $H \leq G$, so the construction above determines an isomorphism of $G$-Mackey functors 
\[\begin{tikzcd}  \underline{\Omega}^G_* \ar[r,"\cong"] & \underline{\pi}_*(\Omega_G). 
\end{tikzcd} \]
\end{theorem}

\section{The Tate square for $\Omega^{C_2}_*$}

In this section we begin our calculation of $\Omega^{C_2}_*$. Since, we've identified the $C_2$-cobordism ring with the coefficients of the $C_2$-spectrum $\Omega_{C_2}$, we may use techniques from equivariant stable homotopy to calculate $\Omega^{C_2}_*$. For background on equivariant stable homotopy theory, see \cite{LMS}. Our primary tool is the Tate square associated to a $C_2$-spectrum $E_{C_2}$, which, at the level of coefficients, has the form 
\[ \begin{tikzcd} 
E^{C_2}_* \ar[r] \ar[d] &E^{\Phi C_2}_* \ar[d] \\
E^{hC_2}_* \ar[r] & E^{tC_2}_*.
\end{tikzcd} \]
Here,  $E^{C_2}_*$, $E^{hC_2}_*$, $E^{\Phi C_2}_*$, and $E^{tC_2}_*$ are called the categorical, homotopy, geometric, and Tate fixed points of $E_{C_2}$, respectively. Note that for $E_{C_2} = \Omega_{C_2}$, we have identified the categorical fixed points $\Omega^{C_2}_*$ of $\Omega_{C_2}$ with the geometrically defined cobordism ring of $C_2$-manifolds, so our notation is consistent.

We refer the interested reader to \cite{GreenleesMay2} for a more expert understanding of the Tate square. For our purposes, it suffices to understand that the other flavors of fixed points in this diagram are often more easily computable than $E^{C_2}_*$, and that if the map
\[
E^{hC_2}_* \oplus E^{\Phi C_2}_* \to E^{tC_2}_*
\]
is surjective, then this square is a pullback (see \cite{Kriz} Lemma 2.1). When $E_{C_2} = \Omega_{C_2}$, we will see by direct computation that this condition is satisfied, so that the Tate square
\[ \begin{tikzcd} 
\Omega^{C_2}_* \ar[r] \ar[d] & \Omega^{\Phi C_2}_* \ar[d] \\
\Omega^{hC_2}_* \ar[r] & \Omega^{tC_2}_*,
\end{tikzcd} \]
is a pullback of rings. We begin our calculation by identifying the homotopy, geometric, and Tate fixed points of $\Omega_{C_2}$.

\begin{itemize}
\item
{\bf Homotopy fixed points:} If $E_{C_2}$ is a $C_2$-spectrum, then the homotopy fixed points of $E_{C_2}$ are defined by 
\[
E^{hC_2}_* = \pi_*^{C_2}(F(EC_{2+},E_{C_2}))
\]
where $F(EC_{2+},E_{C_2})$ denotes the spectrum of maps from the free, non-equivariantly contractible $C_2$-space $EC_2$ to the spectrum $E_{C_2}$. The skeletal filtration of the $C_2$-space $EC_2$ determines a spectral sequence which converges to $E^{hC_2}_*$, called the {\it homotopy fixed point spectral sequence}. If $E_{C_2}$ is a split\footnote{ We say a $C_2$-spectrum $E_{C_2}$ is split if there is a map  of non-equivariant spectra $(E_{C_2})^e \to (E_{C_2})^{C_2}$ such that the composite $(E_{C_2})^e \to (E_{C_2})^{C_2} \to (E_{C_2})^e$
 is homotopic to the identity, Here, $(E_{C_2})^e$ is the non-equivariant spectrum underlying $E_{C_2}$, and $(E_{C_2})^{C_2}$ is the fixed point spectrum of $E_{C_2}$.} $C_2$-spectrum,  then the homotopy fixed point spectral sequence collapses, which results in an isomorphism
 \[
 E^{hC_2}_* \cong H^{-*}(C_2, E_*)
 \]
between the homotopy fixed points of $E_{C_2}$, and the group cohomology of $C_2$ with coefficients in the (trivial) $C_2$-module $E_*$. The $C_2$-spectrum $ \Omega_{C_2}$ is known to be split \cite{GreenleesMay2}, which leads to the identification
\begin{align*}
\Omega^{hC_2}_* & \cong H^{-*}(C_2 ; \Omega_*) \\
 & = \Omega_{*}[[e]],
 \end{align*}
 where $|e| = -1$. The element $e \in \Omega^{hC_2}_*$ is called the {\it Euler class} associated to the sign representation $\sigma$ of $C_2$, and plays an important role in the theory of {\it equivariant formal group laws} as defined in \cite{CGK1}, and in equivariant stable homotopy theory more generally.
 
  \item {\bf Tate fixed points:}
 The Tate fixed points may be obtained from the homotopy fixed points by inverting the Euler class $e$, which gives
 \[
 \Omega^{tC_2}_* = \Omega_*((e)),
 \]
 the ring of Laurent series in $e$ with finitely many negative powers of $e$.

\item {\bf Geometric fixed points:} The geometric fixed points of $\Omega_{C_2}$, which were originally calculated by tom Dieck \cite{tomDieck}, are given by
\[
\Omega^{\Phi C_2}_* = \bigoplus_{k=0}^\infty \Omega_{*-k}(BO(k)).\]
Elements of this ring are cobordism classes of pairs $(F,\xi)$ of a manifold $F$ equipped with a vector bundle $\xi \to F$. The ring $\Omega^{\Phi C_2}_*$ has a simple algebraic description. The $\Omega$ homology of $BO(k)$ is known to be
 \[
\Omega_*(BO(k)) = \Omega_*\{ \beta_{i_1}\dots\beta_{i_k}  : 0 \leq  i_1 \leq \dots \leq i_k\},
 \]
where
\[ \beta_{i_1} \dots \beta_{i_k} = [\mathbb{R}P^{i_1} \times \dots \times \mathbb{R}P^{i_k} , \gamma \times \cdots \times \gamma],\]
and where $\gamma \to \mathbb{R}P^n$ denotes the tautological vector bundle (see, for instance \cite{Milnor}). If we set
 \[
 d_i = \beta_i =  [\mathbb{R}P^{i}, \gamma] \in \Omega_{i}(BO(1)),
 \]
 then 
 \begin{align*}
 \Omega^{\Phi C_2}_* & \cong \bigoplus_{k \geq 0} \Omega_{*-k}(BO(k))\\
 & = \Omega_*[d_0, d_1, d_2, \dots].
 \end{align*}
 \end{itemize}
 
We've identified the geometric, homotopy and Tate fixed points of $\Omega^{C_2}_*$, and our next step is to identify the maps in this diagram. We begin with the maps whose target is $\Omega_*((e))$. As we mentioned before, the map $\Omega_*[[e]] \to \Omega_*((e))$ is localization at $e$, so we turn our attention to the map $\Omega_*[d_0,d_1,\dots] \to \Omega_*((e))$. In order to describe this map, we review the necessary background on formal group laws.

Recall that a formal group law over a commutative ring $A$ is a power series $F(y,z) \in A[[y,z]]$ satisfying an associativity, unitality, and commutativity axiom. If $A$ is an $\mathbb{F}_2$-algebra and the formal group law $F$ satisfies $F(x,x) = 0$, we say $F$ is a $[2]${\it -torsion formal group law} over $A$. These arise in homotopy theory in the following way. Consider the map
\[
\otimes : \mathbb{R}P^\infty \times \mathbb{R}P^\infty \to \mathbb{R}P^\infty
\]
which classifies the tensor product of line bundles. If $E$ is an $MO$-algebra spectrum, then $E^*(\mathbb{R}P^\infty)$ is equal to $E_*[[x]]$, $E^*(\mathbb{R}P^\infty \times \mathbb{R}P^\infty)$ is equal to $E_*[[y,z]]$,  and the image of $x$ under $E^*(\otimes)$ is a $[2]$-torsion formal group law $F(y,z) \in E_*[[y,z]]$ over $E_*$. In particular, this construction yields a $[2]$-torsion formal group law
 \[
 F_{MO}(y,z) = \sum_{i,j \geq 0} a_{i,j} y^i z^j
 \]
 over $MO_*$, and Quillen \cite{Quillen} proved that this is in fact the universal $[2]$-torsion formal group law. Having made the necessary definitions, we may now identify the map $\Omega^{\Phi C_2}_* \to \Omega^{tC_2}_*$.

\begin{lemma}
The composite 
\[ \begin{tikzcd} \phi \colon \Omega_*[d_0,d_1,d_2,\dots] \ar[r,"\cong"] & \Omega^{\Phi C_2}_* \ar[r] & \Omega^{tC_2}_* \ar[r,"\cong"] &   \Omega((e))
\end{tikzcd} 
\]
is given by 
\[
\phi(d_i) = \sum_{j\in \mathbb{Z}} c_{i,j} e^j = \text{coeff}\left(x^j, \frac{1}{F_{MO}(e,x)} \right),
\]
where 
\[
\dfrac{1}{F_{MO}(e,y)} = \sum_{i \geq 0} \sum_{j \in \mathbb{Z}} c_{i,j}y^ie^j \in \Omega_*((e))[[y]]
\]
is the multiplicative inverse of the universal $[2]$-torsion formal group law $F_{MO}(e,y) = \sum a_{i,j} e^i y^j$.
\end{lemma} 

\begin{proof}
Write 
\[
\Omega_*(BO) = \Omega_*[\beta_1,\beta_2,\dots]
\]
where $\beta_i = [\mathbb{R}P^{i}, \gamma]$. 
The map $\phi$ factors as the composite 
\[
\bigoplus_{k \geq 0} \Omega_{*-k}(BO(k)) \to \bigoplus_{k \in \mathbb{Z}} \Omega_{*-k}(BO) \to \bigoplus_{k \in \mathbb{Z}} \Omega_{*-k}(BO)
\]
where the first map is induced by the inclusions $BO(k) \to BO$, and the second map is induced by the endomorphism $\iota:BO \to BO$ which classifies inverses of stable vector bundles. We know that the map 
\[
\bigoplus_{k \in \mathbb{Z}} \Omega_{*-k}(BO) = \Omega_*[e^{\pm 1}, \beta_1 , \beta_2 ,\dots] \to \Omega_*((e))
\]
is defined so that the induced map 
\[
\Omega_*[e^{\pm 1}, \beta_1,\beta_2, \dots ] [[y]] \to \Omega_*((e))[[y]]
\] 
satisfies 
\[
e(1+ \beta_1y+\beta_2y^2 + \dots) \mapsto F_{MO}(e,y),
\]
where $F_{MO}$ is the universal $2$-torsion formal group law. This is the unoriented analogue of a result proved by Kriz in the complex case (see \cite{Kriz}). Set 
\[
\bar{\beta}_i = [\mathbb{R}P^{i} , - \gamma ],
\]
so that the antipode on the Hopf algebra $\Omega_*(BO)$ is given by $\beta_i \mapsto \bar{\beta}_i$. It follows that the map
\[
\Omega_*[e^{\pm 1}, \beta_1,\beta_2, \dots ] [[y]] \to \Omega_*((e))[[y]]
\]
satisfies 
\[
e^{-1}(1 + \bar{\beta}_1y+ \bar{\beta}_2y^2 + \dots)  \cdots \mapsto \dfrac{1}{F_{MO}(e,y)}.
\]
We conclude by observing that $\nu \mid_{\mathbb{R}P^i}^{\mathbb{R}P^{i+1}} = \gamma$, so that $\iota_*(d_i) = e^{-1} \bar{\beta}_i$, and so our map satisfies 
\[
d_0 + d_1y+ d_2y^2 + \dots \mapsto  \dfrac{1}{F_{MO}(e,y)}.
\]

\end{proof}

Next, we need to identify the maps whose domain is $\Omega^{C_2}_*$. We will give a geometric description of each of these. We start with the geometric fixed points $\Omega^{\Phi C_2}_*$, whose elements are cobordism classes of pairs $(F,\xi)$ of a (non-equivariant) manifold $F$ equipped with a vector bundle $\xi \to F$. In these terms, the map 
\[ \begin{tikzcd} 
\Omega^{C_2}_* \ar[r] &  \Omega^{\Phi C_2}_*
\end{tikzcd} \] 
is given by the formula
\[ \begin{tikzcd} 
\left[M\right] \ar[r,mapsto] & \left[M^{C_2} , \nu \mid_{M^{C_2}}^M\right]
\end{tikzcd} \]
 where $M^{C_2}$ is the fixed point submanifold of the $C_2$-manifold $M$, and where $\nu \mid_{M^{C_2}}^M$ is the normal bundle of $M^{C_2}$ in $M$. The map $\Omega^{C_2}_* \to \Omega_*[[e]]$ may be described in terms of the {\it Conner-Floyd operation} $\Gamma$, which we review now (see \cite{connerfloyd} for more details). Let $S(1 +\sigma)$ denote the unit circle in the regular representation $\mathbb{R}^{1+\sigma}$ of $C_2$. If $M$ is a $C_2$-manifold, we define $\Gamma M$ to be the balanced product
 \[
 \Gamma M = M \times_{C_2} S(1+\sigma).
 \]
If we define $\Gamma^n = \Gamma \circ \cdots \circ \Gamma$, then the map from $\Omega^{C_2}_*$ to $\Omega^{hC_2}_*$ is given by 
\[ \begin{tikzcd}
\Omega^{C_2}_* \ar[r] & \Omega_*[[e]]\\
\left[M\right] \ar[r,mapsto] & \sum_{ n \geq 0} \left[ (\Gamma^n M)^e\right] e^n.
\end{tikzcd} \] 
where we write $M^e$ for the underlying non-equivariant manifold of $M$.\footnote{Note that there are two different instances of the symbol $e$ in the expression above. On one hand, we write $e$ for the trivial group, so that the underlying class of a $C_2$-manifold $M$ is $[M^e]$. On the other hand, we write $e$ for the euler class $e \in \Omega^{hC_2}_{-1}$. The meaning of each instance of the symbol $e$ should be clear from context.} This is the unoriented analogue of \cite{Hanke}, Theorem 6.3, and may be verified analogously.

\section{Generators and relations for $\Omega^{C_2}_*$}

In this section, we will complete our calculation of $\Omega^{C_2}_*$ using the Tate square 
 \[ \begin{tikzcd} 
 \Omega^{C_2}_* \ar[r] \ar[d] & \Omega_*[d_0, d_1, d_2,\dots] \ar[d,"\phi"] \\
 \Omega_*[[e]] \ar[r] & \Omega_*((e)).
 \end{tikzcd} \]
Our description of $\Omega^{C_2}_*$ is stated in Theorem \ref{zgraded}. Our proof strategy is as follows. Since $\phi$ maps $d_0$ to the unit $e \in \Omega_*((e))$, we may factor the map $\phi$ as
\[\begin{tikzcd} 
\Omega_*[d_0, d_1,d_2, \dots] \ar[r,"\iota"]  &  \Omega_*[d_0^{\pm 1} , d_1,d_2, \dots] \ar[r, "\psi"]& \Omega_{*}((e))
\end{tikzcd} \]
where $\iota$ is localization at $d_0$. The Tate square for $\Omega^{C_2}_*$ then also factors as 
 \begin{equation} \begin{tikzcd} \label{diagram}
 \Omega^{C_2}_* \ar[r] \ar[d] & \Omega_*[d_0, d_1, d_2,\dots] \ar[d,"\iota"] \\
R \ar[r] \ar[d] & \Omega_*[d_0^{\pm1}, d_1, d_2,\dots] \ar[d,"\psi"] \\
 \Omega_*[[e]] \ar[r] & \Omega_*((e)),
 \end{tikzcd} \end{equation}
 where both squares are pullbacks. We will first calculate the pullback $R$, then use this information to calculate $\Omega^{C_2}_*$. 
 
 \begin{remark} It is worth mentioning that the bottom square in (\ref{diagram}) is isomorphic to the Tate square for the stable cobordism spectrum $MO_{C_2}$. In particular, the ring $R$ is isomorphic to the stable cobordism ring $MO^{C_2}_*$. The relationship between geometric and homotopy theoretic cobordism has been investigated by Sinha \cite{Sinha} in the unoriented, $C_2$-equivariant case, and by Hanke \cite{Hanke} in the complex, torus-equivariant case. 
 
 \end{remark}

\begin{lemma} The square
\[
\begin{tikzcd}
 & d_{i,j} \ar[r,mapsto] & d_ie^{-j} - \sum_{\ell < j} c_{i,\ell} e^{\ell - j} \\
d_{i,j} \ar[d,mapsto] & \dfrac{\Omega_*[e,d_{i,j}: i-1,j \geq 0]}{d_{i,j} - c_{i,j} - ed_{i,j+1}}  \ar[r] \ar[d] & \Omega_*[e^{\pm 1},d_1,d_2,\dots] \ar[d,"\psi"] \\
\sum_{\ell\geq 0} c_{i,j+\ell} e^\ell & \Omega_*[[e]] \ar[r] & \Omega_*((e))
\end{tikzcd} 
\]
is a pullback of rings.
\end{lemma}

\begin{proof}
Following Strickland \cite{Strickland}, it suffices to show that the map 
\[ \begin{tikzcd} 
\Omega_*[e,d_{i,j} ] / (d_{i,j} - c_{i,j} - ed_{i,j+1})  \ar[r] &  
\Omega_*[e^{\pm 1} , d_i]
\end{tikzcd} \]
is an isomorphism after inverting $e$, that 
\[ \begin{tikzcd} 
\Omega_*[e, d_{i,j} ]/ (d_{i,j} - c_{i,j} - ed_{i,j+1}) \ar[r] &  \Omega_*[[e]] 
\end{tikzcd} \]
is an isomorphism after completing at $e$, and that 
\[e \in \Omega_*[e,d_{i,j} ]/(d_{i,j} - c_{i,j} - ed_{i,j+1})\]
is a regular element. The first two assertions are easily verified, and the third follows from the fact that we may write 
\[
\Omega_*[d_{i,j},e]/(d_{i,j} - c_{i,j} - ed_{i,j+1}) = \varinjlim Q_k
\]
where $Q_k = \Omega_*[e,d_{i,k} : i \geq 1]$, and clearly $e$ is a regular element of each of the rings $Q_k$. 
\end{proof}

Having calculated $R$ using the bottom pullback square in \ref{diagram}, we may now use the pullback square
\[ \begin{tikzcd} 
\Omega^{C_2}_* \ar[r] \ar[d] & \Omega_*[d_0 , d_1,d_2,\dots] \ar[d] \\
\Omega_*[e,d_{i,j}]/(d_{i,j} - c_{i,j} - ed_{i,j+1}) \ar[r] & \Omega_*[d_0^{\pm 1} , d_1 ,d_2, \dots ].
\end{tikzcd} \] 
to calculate $\Omega^{C_2}_*$. 

\begin{theorem}\label{zgradedpart} The $C_2$-equivariant geometric bordism ring is given by 
\[
\begin{tikzcd} 
 \Omega^{C_2}_* \cong \dfrac{\Omega_*[d_{i,j} : i - 1, j \geq 0]}{(d_{i,j} - c_{i,j})d_{k,\ell+1} = d_{i,j+1} (d_{k,\ell}  -c_{k,\ell}) }.
 \end{tikzcd} 
 \]
 where $c_{i,j} \in \Omega_{i+j+1}$ is the coefficient of $ e^j y^i$ in $\dfrac{1}{F_{MO}(e,y)}$.
 \end{theorem}

\begin{proof}
First, we claim that 
\[
\Omega^{C_2}_* \to \Omega_*[e,d_{i,j}]/(d_{i,j} - c_{i,j} - ed_{i,j+1}) = R
\]
identifies $\Omega^{C_2}_*$ with the $\Omega_*$ sub-algebra of $R$ generated by $\{d_{i,j} : i-1,j \geq 0\}$. If $f \in R$, then, using the relations in $R$, we may write 
\[
f = f_0(d_{i,j}) + e f_1(d_{i,0} , e)
\]
for some $f_0 \in \Omega_*[d_{i,j}]$ and $f_1 \in \Omega_*[d_{i,0} , e]$. By considering the highest power of $e = d_0^{-1}$ that may occur in $\phi(f)$, we can deduce that $\phi(f)$ lifts to $\Omega_*[d_0 , d_1, d_2 , \dots]$ if and only if $f_1(d_{i,0} , e) = 0$, i.e., if and only if $f$ may be written as $f = f_0(d_{i,j})$. 

We are left  to compute the kernel of the map 
\[
\Omega_*[d_{i,j}] \to \dfrac{\Omega_*[e,d_{i,j} ]}{d_{i,j} - c_{i,j} - e d_{i,j+1}}.
\]
In $R$,  we may write the element $ed_{i,j + 1} d_{k,\ell + 1} $ as
\begin{align*}
(d_{i,j} - c_{i,j})d_{k,\ell + 1}
\end{align*}
or as
\begin{align*}
d_{i,j+1} (d_{k,\ell } - c_{k,\ell}).
\end{align*}
This gives us the relation
\[
(d_{i,j} - c_{i,j})d_{k,\ell + 1}= d_{i,j+1} (d_{k,\ell } - c_{k,\ell})
\]
for any $i,k \geq 1$ and $j,\ell \geq 0$. We claim that these form a complete set of relations for $\Omega^{C_2}_*$. To see this, we may choose an ordering of the generators $\{d_{i,j},e \}$ such that $e$ is the largest element in the ordering. This induces a partial ordering on the set of monomials in the variables $\{d_{i,j}, e\}$. We may then apply Buchberger's algorithm, which allows us to eliminate all instances of the variable $e$. This is a purely algebraic result, and is an immediate consequence of \cite{Jack}, Lemma 7.1.
\end{proof}

We have calculated $\Omega^{C_2}_*$ in terms of certain generators $d_{i,j} \in \Omega^{C_2}_*$. Since every class in $\Omega^{C_2}_*$ is represented by some $C_2$-manifold, it is natural to ask for $C_2$-manifold representatives of the classes $d_{i,j} \in \Omega^{C_2}_*$. We conclude our calculation of $\Omega^{C_2}_*$ by identifying manifold representatives of the classes $d_{i,j}$. For any $n \geq 0$, we define the {\it twisted projective space} $\mathbb{R}P^n_\sigma$ to be the $C_2$-manifold whose underlying space is $\mathbb{R}P^n$, and whose $C_2$-action is given by
\[
 (x_0 : x_1 : \cdots : x_n) \mapsto  (-x_0, x_1 : \cdots : x_n).
\]
For example, the twisted projective space $\mathbb{R}P^1_\sigma$ is $C_2$-equivariantly homeomorphic to the $C_2$-space $S(1+\sigma)$, which is the unit sphere in the regular representation $\mathbb{R}^{1+\sigma}$ of  $C_2$. This evidently bounds the disc $D(1+\sigma)$, so $[\mathbb{R}P^1_\sigma] = 0$ in $\Omega^{C_2}_1$. It turns that if $n \geq 2$, then the class $[\mathbb{R}P^n_\sigma] \in \Omega^{C_2}_n$ is non-zero, and these classes will form part of our generating set. The rest of the generating set will be obtained from the twisted projective spaces $\mathbb{R}P^n_\sigma$ by (iteratively) applying the Conner-Floyd operation $\Gamma$.
\begin{lemma}
For any $i \geq 1$ and $j \geq 0$, we have
\[ d_{i,j} = [\Gamma^{j} \mathbb{R}P^{i+1}_\sigma] \in \Omega^{C_2}_* \]
\end{lemma}
\begin{proof}
Since the map $\Omega^{C_2}_* \to \Omega^{\Phi C_2}_*$ is injective, it suffices to verify the equality in $\Omega^{\Phi C_2}_*$. Recall that the image of $[M] \in \Omega^{C_2}_*$ in $\Omega^{\Phi C_2}_*$ is given by $[M^{C_2},\nu\mid_{M^{C_2}}^M]$.  If $M= \mathbb{R} P^{i+1}_\sigma$, then 
\[
[M^{C_2},\nu\mid_{M^{C_2}}^M] = [\mathbb{R}P^i , \nu \mid_{\mathbb{R}P^i}^{\mathbb{R}P^{i+1}}] + [* , \mathbb{R}^{i+1}] ,
\]
and by definition of the classes $d_i \in \Omega^{\Phi C_2}_*$, this is exactly the element $ d_i - d_0^{i+1}$.
On the other hand, the image of $d_{i,0}$ in $\Omega^{\Phi C_2}_*$ is 
\[
d_i - \sum_{\ell < 0} c_{i,\ell} e^\ell.
\]
But since $d_0 = e^{-1}$, and 
\[
c_{i,\ell} = \begin{cases} -1 & \ell = -i-1\\
0 & 0 > \ell \neq -i-1,\end{cases}
\]
we can deduce that $d_{i,0} = [\mathbb{R}P^{i+1}_\sigma]$. The case $j >0$ follows from the fact that $d_{i,j+1}$ is obtained from $d_{i,j}$ by applying the Conner-Floyd operation $\Gamma$. 
\end{proof}

Combining our previous results, we may state our calculation of $\Omega^{C_2}_*$.

\begin{theorem}\label{zgraded} The $C_2$-equivariant unoriented cobordism ring $\Omega^{C_2}_*$ is generated over $\Omega_*$ by the classes
\[ \begin{tikzcd} 
d_{i,j} = [ \Gamma^{j} \mathbb{R}P^{i+1}_\sigma] &  i \geq 1 \text{ and } j \geq 0.
\end{tikzcd} \]
A complete set of relations is given by 
\[ \begin{tikzcd} 
(d_{i,j} -c_{i,j})d_{k,\ell+1} - d_{i,j+1} (d_{k,\ell} - c_{k,\ell}) & i,k \geq 1 \text{ and } j, \ell \geq 0.
\end{tikzcd} \]
where $c_{i,j} \in \Omega_{i+j+1}$ is the coefficient of $ e^j y^i$ in $\dfrac{1}{F_{MO}(e,y)}$. The restriction
\[
\Omega^{C_2}_* \to \Omega_*
\]
is determined by 
\[
d_{i,j} \mapsto c_{i,j},
\]
and the transfer $\Omega_* \to \Omega^{C_2}_*$ is zero.
\end{theorem}

\section{The extended cobordism ring $\Omega^{C_2}_\diamond$}

We have seen two $C_2$-equivariant analogues of cobordism, namely the geometric cobordism ring $\Omega^{C_2}_*$, and the stable cobordism ring $MO^{C_2}_*$. It turns out that the structure of each of these rings is subsumed by the structure of the {\it extended} cobordism ring $\Omega^{C_2}_\diamond$, which is the subject of the present section. This ring can be viewed from two different perspectives, and we pause to discuss each of these. 

On one hand, the ring $\Omega^{C_2}_\diamond$ has a geometric interpretation. For any $k \geq 0$, let $D(k\sigma)$ denote the unit disc in the orthogonal $C_2$-representation $\mathbb{R}^{k\sigma}$, whose boundary $\partial D(k\sigma) = S(k\sigma)$ is 
 the unit sphere in $\mathbb{R}^{k\sigma}$. Then $\Omega^{C_2}_*(D(k\sigma),S(k\sigma))$ is the cobordism group of $C_2$-manifolds $(M,\partial M)$ equipped with a map to $(D(k\sigma),S(k\sigma))$. This is the standard geometric intepretation of the cobordism homology of a pair of spaces. From this perspective, the extended cobordism ring $\Omega^{C_2}_\diamond$ is the direct sum
 \begin{align*}\Omega^{C_2}_\diamond = \bigoplus_{k\geq 0} \Omega^{C_2}_{*}(D(k\sigma),S(k\sigma)).
 \end{align*}
with product structure specified by
\[
[M \to D(k\sigma)][M' \to D(k'\sigma)] = [M \times M' \to D((k+k')\sigma)].
\]
Here, we have identified the $C_2$-spaces $D(k\sigma) \times D(k'\sigma) \cong D((k+k')\sigma)$.

On the other hand, the ring $\Omega^{C_2}_\diamond$ has a homotopical interpretation: it is a portion of the $RO(C_2)$-graded coefficient ring $\Omega^{C_2}_\star$ of $\Omega_{C_2}$. To see this, we note that if $A \subset X$ is an inclusion of (unbased) $C_2$-spaces, then the Pontrjagin-Thom construction extends to an identification 
\[
\Omega^{C_2}_*(X,A) \cong \pi_*^{C_2}(\Omega_{C_2} \wedge X/A).
\]
If $(X,A)$ is the pair $(D(k\sigma),S(k\sigma))$, then $D(k\sigma)/S(k\sigma) \cong S^{k\sigma}$, the representation sphere associated to the $C_2$-representation $k\sigma$, and the preceding isomorphism takes the form
\begin{align*}
\Omega^{C_2}_*(D(k\sigma),S(k\sigma))&  \cong \pi_*^{C_2}(\Omega_{C_2} \wedge S^{k\sigma})\\
& \cong \Omega^{C_2}_{*-k\sigma}.
\end{align*}
This prespective allows us to view
\[
\Omega^{C_2}_\diamond = \bigoplus_{k = 0}^{\infty}  \Omega^{C_2}_{* - k \sigma},
\]
as a subring of the $RO(C_2)$-graded coefficients of $\Omega_{C_2}$. This perspective is useful because, for instance, it affords us access to long exact sequences in $RO(C_2)$-graded homotopy groups induced by cofiber sequences of $C_2$-spectra. We refer the interested reader to \cite{LMS} for further information about $RO(C_2)$-graded homotopy groups, and to chapter XV of \cite{alaska} for further details regarding the equivariant Pontrjagin-Thom construction in the relative case.

There are two especially important examples of relative $C_2$-manifolds. The first is the {\it shifted Euler class}
\[
a = [* \to D(\sigma)] \in \Omega^{C_2}_{-\sigma},
\]
which is represented by the inclusion of the origin in the unit disc $D(\sigma)$. The second is the {\it orientation class}
\[
u = [D(\sigma) \to D(\sigma)] \in \Omega^{C_2}_{1-\sigma},
\]
which represented by the identity map on $D(\sigma)$. It turns out that the extended cobordism ring $\Omega^{C_2}_\diamond$ is generated over $\Omega^{C_2}_*$ by the classes $a$ and $u$. In the following lemma, we identify the relationship between geometric classes, the shifted Euler class, and the orientation class.

\begin{lemma}
If $M$ is a $C_2$-manifold, then the equation
\[
u[M] = u[M^e] + a[\Gamma M]
\]
holds in $\Omega^{C_2}_{*-\sigma}$.
\end{lemma}
\begin{proof}
Recall that elements of $\Omega^{C_2}_{*-\sigma}$ are represented by $C_2$-manifolds $(X,\partial X)$ equipped with a map to $(D(\sigma),S(\sigma))$. The classes $u[M]$, $u[M^e]$, and $a[\Gamma M]$ are represented geometrically as follows:
\begin{align*}
u[M] & = [M \times D(\sigma) \to D(\sigma)]\\
u[M^e] & = [M^e \times D(\sigma) \to D(\sigma)] \\
a[\Gamma M] & = [\Gamma M \to *  \to D(\sigma)].
\end{align*} 
Since the map 
\[ \begin{tikzcd} \Omega^{C_2}_{*-\sigma} \ar[r] & \Omega^{\Phi C_2}_{*-\sigma} \ar[r,swap,"\cong"] \ar[r,"a^{-1}"] &  \Omega^{\Phi C_2}_*\end{tikzcd} \]
is injective, it suffices to prove the relation in the fixed point bordism ring $\Omega^{\Phi C_2}_*$. The image of the classes $u[M]$ and $u[M^e]$ in $\Omega^{\Phi C_2}_*$ are given by $[M^{C_2}, \nu \mid_{M^{C_2}}^M \oplus \mathbb{R}]$ and $[M^e, \mathbb{R}] $. The image of the class $a[\Gamma M]$ in $\Omega^{\Phi C_2}_*$ is
\[
[M^{C_2} \sqcup M^e , \nu \mid_{M^{C_2}}^M \oplus \mathbb{R} \sqcup  \mathbb{R}] = [M^{C_2} , \nu \mid_{M^{C_2}}^M \oplus \mathbb{R}] + [ M^e , \mathbb{R}] .
\]
\end{proof}

It turns out that this relation completely determines the structure of the extended bordism ring $\Omega^{C_2}_\diamond$ as an algebra over $\Omega^{C_2}_*$.
\begin{theorem} As an algebra over the $C_2$-cobordism ring $\Omega^{C_2}_*$, the extended cobordism ring $\Omega^{C_2}_\diamond$ is generated by the classes $a \in \Omega^{C_2}_{-\sigma}$ and $u \in \Omega^{C_2}_{1-\sigma}$. A complete set of relations is given by 
\[
u[M] = u[M^e] + a[\Gamma M]
\]
for each $[M] \in \Omega^{C_2}_*$.
\end{theorem}

\begin{proof}
It suffices to prove the $\Omega_*$-module isomorphism 
\begin{align*}
\Omega^{C_2}_{*-n\sigma} & \cong \dfrac{\Omega^{C_2}_* \{ a^n , a^{n-1} u , \dots , au^{n-1}, u^n \} }{a^{i} u^{n-i} [M] = a^{i}u^{n-i} [M^e] + a^{i+1} u^{n-i-1} [\Gamma M] } \hspace{0.5in}  1 \leq i \leq n 
\end{align*}
for every $n \geq 0$. We prove the claim by induction on $n \geq 0$, with the case $n = 0$ holding trivially. Suppose that we have proved the claim for some $n \geq 0$. We apply $\Omega^{C_2}_{*-n\sigma}(-)$ to the cofiber sequence 
\[
C_{2+} \to S^0 \to S^\sigma
\]
to obtain the exact sequence 

\[ \begin{tikzcd} 
\cdots \ar[r] & \Omega_{*+1-(n+1)\sigma} \ar[r,"\text{tr}_e^{C_2}"] & \Omega^{C_2}_{*-n\sigma} \ar[r,"a"] &   \Omega^{C_2}_{*-(n+1)\sigma} \ar[r,"\text{res}^{C_2}_e"] &  \Omega_{*-(n+1)\sigma} \ar[r] &  \cdots.
\end{tikzcd}  \] 
By induction we may assume that the map $\text{tr}_e^{C_2}$ above  is zero, and the map $\text{res}^{C_2}_e$ above is surjective, which implies that our exact sequence is short:

\[ \begin{tikzcd} 
0 \ar[r] & \Omega^{C_2}_{*-n\sigma} \ar[r,"a"] &   \Omega^{C_2}_{*-(n+1)\sigma} \ar[r] &  \Omega_{*-(n+1)\sigma} \ar[r] &  0.
\end{tikzcd}  \] 
The righthand term $\Omega_{*-(n+1)\sigma}$ is a free $\Omega_*$-module, and the short exact sequence above is split by $u^{n+1}$, which leads to the isomorphism
\begin{align*}
\Omega^{C_2}_{*-(n+1)\sigma} & \cong  \Omega_*\{u^{n+1}\} \oplus \Omega^{C_2}_{*-n\sigma}\{a\}  \\
& \cong \dfrac{\Omega^{C_2}_* \{ a^{n+1} , a^{n} u , \dots , au^{n}, u^{n+1} \} }{a^{i} u^{n+1-i} [M] = a^{i}u^{n+1-i} [M^e] + a^{i+1} u^{n-i} [\Gamma M] } \hspace{0.5in}  1 \leq i \leq n+1
\end{align*}
Taking the direct sum over $n \geq 0$ yields the desired result.
\end{proof}

Combining the previous result with Theorem \ref{zgraded}, we obtain the following description of  $\Omega^{C_2}_\diamond$.

\begin{theorem}\label{goodrange}
The extended $C_2$-bordism ring $\Omega^{C_2}_\diamond$ is generated over $\Omega_*$ by the classes
\[ \begin{tikzcd} 
d_{i,j} = [ \Gamma^{j} \mathbb{R}P^{i+1}_\sigma] &  i \geq 1 \text{ and } j \geq 0,
\end{tikzcd} \]
together with the shifted Euler class
\[
a = [ * \to D(\sigma)]
\]
and the orientation class
\[
u = [D(\sigma) \to D(\sigma)].
\]
A complete set of relations is given by 
\[ \begin{tikzcd} 
(d_{i,j} -c_{i,j})d_{k,\ell+1} - d_{i,j+1} (d_{k,\ell} - c_{k,\ell}) & i,k \geq 1 \text{ and } j, \ell \geq 0,\\
u(d_{i,j} - c_{i,j}) - a d_{i,j+1} & i \geq 1 \text{ and } j \geq 0,
\end{tikzcd} \]
where $c_{i,j} \in \Omega_{i+j+1}$ is the coefficient of $ e^j y^i$ in $\dfrac{1}{F_{MO}(e,y)}$. The restriction
\[
\Omega^{C_2}_\diamond \to \Omega_\diamond = \Omega_*[u]
\]
is determined by 
\begin{align*}
d_{i,j} & \mapsto c_{i,j}\\
a & \mapsto 0,\\
u & \mapsto u.
\end{align*}
and the transfer $\Omega_\diamond \to \Omega^{C_2}_\diamond$ is zero.

\end{theorem}





\bibliographystyle{amsalpha}

\end{document}